\documentclass[12pt,twoside]{amsart}
\usepackage[all]{xy}
\usepackage{graphicx}
\usepackage{amsmath}

\usepackage{longtable}

\title[Anti-flips of the blow-ups of 
the projective spaces]
{Anti-flips of the blow-ups of 
the projective spaces at torus invariant points}
\author{Hiroshi Sato} 
\author{Shigehito Tsuzuki} 
\subjclass[2020]{Primary 14M25; Secondary 14E05, 14J45.}
\date{2023/5/6, version 1.0}
\keywords{toric varieties, Fano varieties, 
blow-ups, anti-flips}
\address{Department of Applied Mathematics, Faculty of Sciences, 
Fukuoka University, 8-19-1, Nanakuma, Jonan-ku, Fukuoka 814-0180, Japan}
\email{hirosato@fukuoka-u.ac.jp}
\address{Department of Applied Mathematics, Faculty of Sciences, 
Fukuoka University, 8-19-1, Nanakuma, Jonan-ku, Fukuoka 814-0180, Japan}
\email{sd210005@cis.fukuoka-u.ac.jp}

\makeatletter
    
    \@addtoreset{equation}{section}
\makeatother

\newcommand{\G}[0]{{\operatorname{G}}}

\newtheorem{thm}{Theorem}[section]
\newtheorem{lem}[thm]{Lemma}

\newtheorem{prop}[thm]{Proposition}

\theoremstyle{definition}
\newtheorem{ex}[thm]{Example}
\newtheorem{defn}[thm]{Definition}

\newtheorem{rem}[thm]{Remark}
\newtheorem*{ack}{Acknowledgments}       

\begin{document}
\bibliographystyle{amsalpha+}

\begin{abstract}
We explicitly construct the smooth toric Fano variety which is 
isomorphic to 
the blow-up of the projective space at torus invariant points 
in codimension one by anti-flips. 
\end{abstract}

\thispagestyle{empty}

\maketitle

\tableofcontents
\section{Introduction} 
The blow-up of the projective plane $\mathbb{P}^2$ at $1$, $2$ or $3$ 
torus invariant points 
are isomorphic to 
the Hirzebruch surface $F_1$ of degree $1$, the del Pezzo surface $S_7$ of 
degree $7$ or the del Pezzo surface $S_6$ of degree $6$, respectively. 
As is well known, they are {\em Fano} varieties, 
that is, their anti-canonical divisors are ample. 
For $d\ge 3$, let $B^d_n$ be the blow-up of $\mathbb{P}^d$  at $n$ torus invariant points. 
Then, $B^d_1$ is a Fano variety, while $B^d_n$ is not a Fano variety for $n\ge 2$ 
(see e.g. \cite{bonavero}). 

In this paper, we construct the smooth Fano variety $\widetilde{B}^d_n$ which is 
birationally equivalent to $B^d_n$ 
by a finite succession of {\em anti-flips}. 
For this construction, we investigate the {\em primitive relations} for toric anti-flips. 
The construction of $\widetilde{B}^d_n$ is 
a generalization for the theory of pseudo-symmetric and symmetric 
toric varieties by Ewald \cite{ewald1} and Voskresenskij-Klyachko \cite{voskre}. 
The following is the main theorem of this paper. 

\begin{thm}[Theorem \ref{kekka}]\label{introkekka}
Let $B^d_n$ be the blow-up of $\mathbb{P}^d$ at $n$ torus invariant points. 
Suppose that $d\ge 3$ and $n\ge 2$. 

If $2n-1<d$ or $d$ is even, then there exists a finite succession 
$B^d_n\dashrightarrow \widetilde{B}^d_n$ 
of anti-flips such that $\widetilde{B}^d_n$ is a smooth toric Fano variety. 
More precisely, if $2n-1<d$, then $\widetilde{B}^d_n$ has a 
$(\mathbb{P}^1)^n$-bundle structure over $\mathbb{P}^{d-n}$, while 
otherwise every extremal ray of the Kleiman-Mori cone of $\widetilde{B}^d_n$ is of small type. 

If $2n-1\ge d$ and $d$ is odd, then there 
does not exist a {\em smooth} Fano variety which is isomorphic to 
$B^d_n$ in codimension one. 
\end{thm}

This paper is organized as follows: 
Section \ref{prepre} is devoted to the calculation of anti-flips 
by using the notion of primitive relations. This will be useful 
for the birational geometry of toric varieties. 
In Section \ref{mainsec}, we prove Theorem \ref{introkekka}. 
In each case, 
we can explicitly describe the number of anti-flips 
to obtain the smooth toric Fano variety (see Theorem \ref{kekka}).

\begin{ack}
The first author was partly supported by JSPS KAKENHI 
Grant Number JP18K03262. 
The authors thank the referee very much for many useful comments. 
They also thank Professor Osamu Fujino, who kindly answered 
their questions about minimal model theory. 
\end{ack}

\section{Preliminary}\label{prepre}

In this section, we quickly review the notion of 
primitive collections and relations for toric varieties 
introduced by Batyrev \cite{bat1} (see also 
\cite{bat2} and \cite{sato1}). They are convenient to 
describe the fan associated to a smooth complete 
toric variety. By using them, we can explicitly calculate 
some important operations in the birational geometry 
like blow-ups, blow-downs and anti-flips. 
For the basic theory of the 
toric geometry, see \cite{cls}, \cite{fulton} 
and \cite{oda}. Moreover, for the toric Mori theory, 
see \cite{fujino-sato}, \cite{matsuki} 
and \cite{reid}. 
We will work over an algebraically closed field $K=\overline{K}$. 

\smallskip

Let $X=X_{\Sigma}$ be the smooth projective 
toric $d$-fold associated to a fan $\Sigma$ 
in $N:=\mathbb{Z}^d$. Put
\[
\G(\Sigma):=\{\mbox{the primitive generators 
for $1$-dimensional cones in }\Sigma\}
\subset N.
\]
There is a one-to-one correspondence between 
$\G(\Sigma)$ and the set of torus invariant prime divisors on $X$. 
In particular, for another smooth projective toric $d$-fold $X'=X_{\Sigma'}$, 
if $\G(\Sigma)=\G(\Sigma')$, then $X$ and $X'$ are isomorphic 
in codimension one. 

The following notion is very important 
for our theory. 
\begin{defn}
A non-empty subset $P\subset\G(\Sigma)$ is a 
{\em primitive collction} of $\Sigma$ (or $X$) 
if 
\begin{enumerate}
\item 
$P$ does not generate a cone in $\Sigma$, while 
\item 
$P\setminus\{u\}$ generates a cone in 
$\Sigma$ for any $u\in\G(\Sigma)$. 
\end{enumerate}
For a primitive collection $P=\{u_1,\ldots,u_l\}$, 
there exists a unique cone $\sigma(P)\in\Sigma$ which 
contains $u_1+\cdots+u_l$ in its relative 
interior. Let $\{v_1,\ldots,v_m\}
\subset\G(\Sigma)$ be 
the generators for $\sigma(P)$ (in particular, $m$ is 
the dimension of $\sigma(P)$). Then, we have 
a linear relation
\[
u_1+\cdots+u_l=a_1v_1+\cdots+a_mv_m,
\]
where $a_1,\ldots,a_m$ are positive integers. 
We call this relation the {\em primitive relation} 
for $P=\{u_1,\ldots,u_l\}$. 
\end{defn}

It is well known that 
for any primitive collection of $X=X_\Sigma$, 
we can associate a numerical $1$-cycle on $X$ by 
using its primitive relation (see e.g. \cite{bat1}). 
In particular, the numerical $1$-cycles 
associated to the primitive collections of $X$ 
generate the Kleiman-Mori cone ${\rm NE}(X)$, 
which is always polyhedral when $X$ is a projective toric variety. 
So, we say that a primitive collection (or a primitive relation) 
is {\em extremal} if the associated numerical $1$-cycle 
generates an extremal ray $R\subset{\rm NE}(X)$. 
Thus, we obtain the extremal contraction $\varphi_R: X\to\overline{X}$ 
associated to an extremal primitive relation 
$u_1+\cdots+u_l=a_1v_1+\cdots+a_mv_m$. 
For the type of $\varphi_R$, the following hold: 
\begin{itemize}
\item 
If $m=0$, then $\varphi_R: X\to\overline{X}$ is a 
{\em Mori fiber space}. In this case, $\varphi_R$ is nothing but 
a $\mathbb{P}^{l-1}$-bundle structure over $\overline{X}$. 
\item 
If $m=1$, then $\varphi_R: X\to\overline{X}$ is a 
{\em divisorial contraction}. Moreover, if $a_1=1$, then 
$\varphi_R$ is a blow-up of $\overline{X}$ along a 
$(d-l)$-dimensional torus invariant subvariety. 
\item 
If $m\ge 2$, then $\varphi_R: X\to\overline{X}$ is a 
{\em small contraction}. Moreover, 
if $l-(a_1+\cdots+a_m)>0$ (resp. $< 0$, $=0$), then $\varphi_R$ is 
a {\em flipping} (resp. {\em anti-flipping}, {\em flopping}) contraction. 
For a flipping (resp. anti-flipping, flopping) contraction $\varphi_R:X\to\overline{X}$, 
we can construct a {\em flip} (resp. {\em anti-flip}, {\em flop})
 \[
\xymatrix{
X\ar[dr]_{\varphi_R}\ar@{-->}[rr]&& X^+\ar[dl]^{\varphi^+_R}\\
&\overline{X}&
}
\]
by the toric Mori theory. An anti-flip is the inverse operation of a flip, that is, 
if $X\dashrightarrow X^+$ 
is a flip, then its inverse rational map $X^+\dashrightarrow X$ is an anti-flip. 
\end{itemize}

We should remark that $\Sigma$ can be 
recovered by all the primitive relations of $\Sigma$. Namely, 
we can describe a fan by giving all the primitive relations of it. 


For blow-ups, the primitive collections can be 
calculated as follows:

\begin{prop}[\cite{sato1}, Theorem 4.3]\label{blowupprop}
Let $X=X_\Sigma$ be a smooth projective toric 
variety and $X'\to X$ be the blow-up with 
respect to an $l$-dimensional cone $\langle u_1,\ldots,u_l\rangle$ 
in $\Sigma$, 
where $\{u_1,\ldots,u_l\}\subset\G(\Sigma)$ $($remark that $l$ is the codimension of 
the center of the blow-up$)$. Put $v:=u_1+\cdots+u_l$. Then, the 
primitive collections of $X'$ are 
\begin{enumerate}
\item 
$\{u_1,\ldots,u_l\}$ $($whose primitive relation is $u_1+\cdots+u_l=v)$, 
\item 
any primitive collection $P$ in $\Sigma$ such that 
$\{u_1,\ldots,u_l\}\not\subset P$ and 
\item 
$(P\setminus\{u_1,\ldots,u_l\})\cup\{v\}$ 
for any primitive collection $P$ of  $\Sigma$ such that 
$P\setminus\{u_1,\ldots,u_l\}$ is a minimal element in 
\[
\left\{Q\setminus\{u_1,\ldots,u_l\}\,
\left|\,Q\mbox{ is a primitive collection of }
\Sigma,\ 
Q\cap\{u_1,\ldots,u_l\}\neq\emptyset\right.\right\}.
\]
\end{enumerate}
\end{prop}

Conversely, we can calculate the primitive collections of a blow-down of 
a smooth projective toric variety. 

\begin{prop}[\cite{sato1}, Corollary 4.9]\label{blowdownprop}
Let $X=X_{\Sigma}$ be a smooth projective toric variety 
and $X\to \overline{X}$ the blow-down with 
respect to an extremal primitive relation
\[
u_1+\cdots+u_l=v
\]
of $X$. 
Then, the primitive collections of $\overline{X}$ 
are 
\begin{enumerate}
\item any primitive collection $P$ 
of $\Sigma$ 
such that $P\neq\{u_1,\ldots,u_l\}$ and 
$v\not\in P$, and 
\item 
$(P\setminus\{v\})\cup\{u_1,\ldots,u_l\}$ 
for any primitive collection $P$ of $\Sigma$ 
such that $v\in P$ and $(P\setminus \{v\})\cup 
S$ is not a primitive collection of $\Sigma$ for 
any proper subset $S\subset \{u_1,\ldots,u_l\}$. 
\end{enumerate}
\end{prop}

By combining Propositions \ref{blowupprop} and \ref{blowdownprop}, 
we obtain the following. This theorem is essential for the calculations in 
Section \ref{mainsec}. 

\begin{thm}\label{mainthm} 
Let $X=X_\Sigma$ be a smooth projective toric variety and 
$X\dashrightarrow X^+=X_{\Sigma^+}$ the anti-flip with respect to an 
extremal primitive relation 
\[
u_1+\cdots+u_l=v_1+\cdots+v_m\ (l<m)
\]
of $\Sigma$, where $\{u_1,\ldots,u_l,v_1,\ldots,v_m\}\subset\G(\Sigma)$. 
Then, $X^+$ is also a {\em smooth} projective toric variety 
and the primitive collections of $\Sigma^+$ are 
\begin{enumerate}
\item $\{v_1,\ldots,v_m\}$ whose 
primitive relation is 
\[
v_1+\cdots+v_m=u_1+\cdots+u_l,
\]
\item any primitive collection $P$ of $\Sigma$ such that 
$\{v_1,\ldots,v_m\}\not\subset P$ and $P\neq\{u_1,\ldots,u_l\}$, and 
\item 
$\left(P\setminus\{v_1,\ldots,v_m\}\right)\cup\{u_1,\ldots,u_l\}$ 
for any primitive collection $P$ of $\Sigma$ such that 
$P\setminus\{v_1,\ldots,v_m\}$ is a 
minimal element in 
\[
\left\{P\setminus\{v_1,\ldots,v_m\}\,\left|\,P\mbox{ is a primitive collection of }\Sigma,\ 
P\cap\{v_1,\ldots,v_m\}\neq\emptyset\right.\right\}
\]
and $(P\setminus\{v_1,\ldots,v_m\})\cup S$ does not contain a primitive collection for any proper subset 
$S\subset\{u_1,\ldots,u_l\}$. 
\end{enumerate}
\end{thm}
\begin{proof}
$X^+$ is obtained by blowing-up $X$ along 
the torus invariant subvariety associated to 
the cone $\langle v_1,\ldots,v_m\rangle$ 
and by blowing-down with respect to the extremal primitive relation 
\[
u_1+\cdots+u_l=v,
\]
where $v:=v_1+\cdots+v_m$. Therefore, we can apply Propositions 
\ref{blowupprop} and \ref{blowdownprop}.
\end{proof}

\begin{rem}\label{flopdemo}
Obviously, Theorem \ref{mainthm} is valid for smooth 
flips (the case where $l>m$) and smooth flops (the case where $l=m$) as well. 
\end{rem}

Low-dimensional examples for this calculation are given in Section \ref{mainsec} 
(see Examples \ref{4dimexample} and \ref{3dimexample}). 


We end this section by giving a characterization of 
{\em Fano} varieties using the notion of 
primitive relations for the reader's convenience:

\begin{prop}[see e.g. \cite{bat2}]\label{fanohantei}
Let $X=X_\Sigma$ be a smooth projective toric 
variety. Then $X$ is a {\em Fano} variety $($resp. {\em weak Fano} variety$)$ if and only if 
for any primitive relation
\[
u_1+\cdots+u_l=a_1v_1+\cdots+a_mv_m
\]
of $\Sigma$, $l-(a_1+\cdots+a_m)>0$ $($resp. $\ge 0$$)$ holds. 
We call $l-(a_1+\cdots+a_m)$ the {\em degree} of 
the primitive collection $($or relation$)$. 
\end{prop}

\section{Blow-ups and anti-flips}\label{mainsec}
First, we give the description of the fans associated to the projective spaces 
and their blow-ups at torus invariant points. 
We will use this notation throughout this section. 

For any natural number $d$, the $d$-dimensional projective space $\mathbb{P}^d$ is 
the simplest complete toric $d$-fold whose fan $\Sigma^d_0$ is described as follows: 
Let $\{e_1,\ldots,e_d\}$ be the standard basis for $N:=\mathbb{Z}^d$, and put 
$x_1:=e_1,\ldots,x_d:=e_d,x_{d+1}:=-(e_1+\cdots+e_d)$. 
Then, $\Sigma^d_0$ has the unique primitive collection with 
the following primitive relation 
\[
x_1+\cdots+x_{d+1}=0,
\]
where $\G(\Sigma^d_0)=\{x_1,\ldots,x_{d+1}\}$. Namely, 
$\mathbb{P}^d$ can be expressed by 
this only one simple equality. 

Suppose that $d\ge 2$, and 
let $\pi:B^d_n\to \mathbb{P}^d$ be the blow-up of $\mathbb{P}^d$ 
at $n$ torus invariant points for $1\le n\le d+1$ and 
$\Sigma^d_n$ the fan associated to $B^d_n$. Here, we should remark that 
$\mathbb{P}^d$ has exactly $d+1$ torus invariant points. 
We may assume $n\ge 2$, since $B^d_1$ itself is a Fano manifold. 
By using Proposition \ref{blowupprop} $n$ times, we obtain the following.
\begin{prop}\label{blowupprim}
The primitive relations 
of $\Sigma^d_n$ are 
\[
x_i+y_i=0\ (1\le i\le n),\ x_1+\cdots+\check{x}_i+\cdots+x_{d+1}=y_i\ (1\le i\le n)\mbox{ and }
\]
\[
y_i+y_j=x_1+\cdots+\check{x}_i+\cdots+\check{x}_j+\cdots+x_{d+1}\ (1\le i<j\le n),
\]
where $\G(\Sigma^d_n)=\{x_1,\ldots,x_{d+1},y_1,\ldots,y_n\}$. 
In particular, $\Sigma^d_n$ has exactly $\frac{n(n+3)}{2}$ primitive collections. 
\end{prop}

The main purpose of this paper is 
to construct the smooth toric Fano variety $\widetilde{B}^d_n$ 
associated to the fan $\widetilde{\Sigma}^d_n$ 
such that $\G(\widetilde{\Sigma}^d_n)=\G(\Sigma^d_n)$, that is,  
$\widetilde{B}^d_n$ and $B^d_n$ are isomorphic in 
codimension one. 
$B^2_2$ and $B^2_3$ themselves are del Pezzo surfaces, 
so we assume $d\ge 3$. Proposition \ref{fanohantei} tells us that 
$B^d_n$ is not a Fano variety for $d\ge 3$. We use the notation 
$\mathrm{C}_{n,r}=\frac{n!}{r!(n-r)!}$ for $1\le r\le n$. 
\begin{lem}\label{induction}
Let $1\le r\le n-1$. 
Suppose that $2r+1<d$ and 
there exists a smooth projective toric $d$-fold $B^d_{n,r}$ 
associated to the fan $\Sigma^d_{n,r}$ such that the primitive relations 
of $\Sigma^d_{n,r}$ are 
\[
(\mathrm{I})\ x_i+y_i=0\ (1\le i\le n),
\]
\[
(\mathrm{I\hspace{-.01em}I})\ \sum_{x\in\{x_1,\ldots,x_{d+1}\}\setminus\{x_{i_1},\ldots,x_{i_r}\}}x=y_{i_1}+\cdots+y_{i_r}\ (1\le i_1<\cdots<i_r\le n)\mbox{ and }
\]
\[
(\mathrm{I\hspace{-.01em}I\hspace{-.01em}I})\ y_{j_1}+\cdots+y_{j_{r+1}}=
\sum_{x\in\{x_1,\ldots,x_{d+1}\}\setminus\{x_{j_1},\ldots,x_{j_{r+1}}\}}x\ 
(1\le j_1<\cdots<j_{r+1}\le n),
\]
where $\G(\Sigma^d_{n,r})=\G(\Sigma^d_n)=\{x_1,\ldots,x_{d+1},y_1,\ldots,y_n\}$. 
Then, there exists a sequence of smooth anti-flips
\[
B^d_{n,r}=:B^d_{n,r}(0)\dashrightarrow B^d_{n,r}(1)
\dashrightarrow\cdots\dashrightarrow B^d_{n,r}(\mathrm{C}_{n,r+1})
=:B^d_{n,r+1}
\]
such that the primitive relations of 
the fan $\Sigma^d_{n,r+1}$ associated to $B^d_{n,r+1}$ are
\[
x_i+y_i=0\ (1\le i\le n),
\]
\[
\sum_{x\in\{x_1,\ldots,x_{d+1}\}\setminus\{x_{j_1},\ldots,x_{j_{r+1}}\}}x=
y_{j_1}+\cdots+y_{j_{r+1}}\ (1\le j_1<\cdots<j_{r+1}\le n)\mbox{ and }
\]
\[
y_{k_1}+\cdots+y_{k_{r+2}}=
\sum_{x\in\{x_1,\ldots,x_{d+1}\}\setminus\{x_{k_1},\ldots,x_{k_{r+2}}\}}x\ 
(1\le k_1<\cdots<k_{r+2}\le n),
\]
where $\G(\Sigma^d_{n,r+1})=\G(\Sigma_n)=\{x_1,\ldots,x_{d+1},y_1,\ldots,y_n\}$. 
\end{lem}
\begin{proof}
First, we remark that $2r+1<d$ means that the degrees of the primitive relations in 
$(\mathrm{I\hspace{-.01em}I\hspace{-.01em}I})$ are negative. 
Moreover, they are extremal by the symmetry of the fan $\Sigma^d_{n,r}$, and 
the associated extremal contractions are anti-flipping contractions. 
In particular, $B^d_{n,r}$ is not a Fano variety. 

Take a primitive relation 
\[
(\star)\ y_{s_1}+\cdots+y_{s_{r+1}}=
\sum_{x\in\{x_1,\ldots,x_{d+1}\}\setminus\{x_{s_1},\ldots,x_{s_{r+1}}\}}x
\]
in $(\mathrm{I\hspace{-.01em}I\hspace{-.01em}I})$, and let 
\[
B^d_{n,r}(0):=B^d_{n,r}\dashrightarrow B^d_{n,r}(1)
\]
be the associated anti-flip. 
Put $\Sigma^d_{n,r}(1)$ be the fan 
associated to $B^d_{n,r}(1)$. 
By $(2)$ in Theorem \ref{mainthm}, the primitive relations of $\Sigma^d_{n,r}$ in $(\mathrm{I})$ 
are also primitive relations of $\Sigma^d_{n,r}(1)$. 
Also, $(2)$ in Theorem \ref{mainthm} says that the primitive relations in 
$(\mathrm{I\hspace{-.01em}I\hspace{-.01em}I})$ other than 
the above primitive relation $(\star)$ are  
also primitive relations of $\Sigma^d_{n,r}(1)$, while 
$\{x_1,\ldots,x_{d+1}\}\setminus\{x_{i_1},\ldots,x_{i_r}\}$ 
(which is a primitive collection of $\Sigma^d_{n,r}$ in $(\mathrm{I\hspace{-.01em}I})$) 
is a primitive collection of 
$\Sigma^d_{n,r}(1)$ if and only if 
it does not contain $\{x_1,\ldots,x_{d+1}\}\setminus\{x_{s_1},\ldots,x_{s_{r+1}}\}$, 
that is, $\{x_{i_1},\ldots,x_{i_r}\}\not\subset \{x_{s_1},\ldots,x_{s_{r+1}}\}$. 
On the other hand, 
\[
\sum_{x\in\{x_1,\ldots,x_{d+1}\}\setminus\{x_{s_1},\ldots,x_{s_{r+1}}\}}x=
y_{s_1}+\cdots+y_{s_{r+1}}
\]
is of course a new primitive relation of $\Sigma^d_{n,r}(1)$. 
Moreover, $(3)$ in Theorem \ref{mainthm} tells us that we have another new 
primitive collection $\{y_i,y_{s_1},\ldots,y_{s_{r+1}}\}$ of $\Sigma^d_{n,r}(1)$ 
if $y_i\not\in\{y_{s_1},\ldots,y_{s_{r+1}}\}$ 
and $\{y_i,y_{s_1},\ldots,y_{s_{r+1}}\}$ contains no primitive collection of $\Sigma^d_{n,r}$ 
other than $\{y_{s_1},\ldots,y_{s_{r+1}}\}$ since $\{x_i,y_i\}$ is a primitive collection of 
$\Sigma^d_{n,r}$. 
We should remark that a primitive collection in $(\mathrm{I\hspace{-.01em}I})$ 
has the non-empty intersection with 
$\{x_1,\ldots,x_{d+1}\}\setminus\{x_{s_1},\ldots,x_{s_{r+1}}\}$, 
however it does not fulfill the condition in $(3)$ in Theorem \ref{mainthm}. 
Thus, we obtain the primitive relations 
\[
x_i+y_i=0\ (1\le i\le n),
\]
\[
\sum_{x\in\{x_1,\ldots,x_{d+1}\}\setminus\{x_{i_1},\ldots,x_{i_r}\}}x=y_{i_1}+\cdots+y_{i_r}
\ (1\le i_1<\cdots<i_r\le n), 
\]
\[
y_{j_1}+\cdots+y_{j_{r+1}}=
\sum_{x\in\{x_1,\ldots,x_{d+1}\}\setminus\{x_{j_1},\ldots,x_{j_{r+1}}\}}x\ 
\left((j_1,\ldots,j_{r+1})\neq(s_1,\ldots,s_{r+1})\right)\mbox{ and }
\]
\[
\sum_{x\in\{x_1,\ldots,x_{d+1}\}\setminus\{x_{s_1},\ldots,x_{s_{r+1}}\}}x=
y_{s_1}+\cdots+y_{s_{r+1}}
\]
of $\Sigma^d_{n,r}(1)$. We remark that in this first case, 
any primitive relation of $\Sigma^d_{n,r}$ in  $(\mathrm{I\hspace{-.01em}I})$ 
does not vanish, while the number of new primitive relations is only one. 
Continuously, by doing the anti-flip with respect to 
a primitive relation in $(\mathrm{I\hspace{-.01em}I\hspace{-.01em}I})$ one by one, 
we obtain a sequence
\[
B^d_{n,r}(0)\dashrightarrow B^d_{n,r}(1)
\dashrightarrow\cdots\dashrightarrow B^d_{n,r}(\mathrm{C}_{n,r+1})
\]
of anti-flips. In each step, we can calculate the primitive relations of 
the anti-flip with the same rules as the first case 
$B^d_{n,r}(0)\dashrightarrow B^d_{n,r}(1)$. 
Eventually, all the primitive relations in $(\mathrm{I\hspace{-.01em}I})$ vanish, 
and $\{k_1,\ldots,k_{r+2}\}$ becomes a primitive collection for any 
$1\le k_1<\cdots<k_{r+2}\le n$.  
This shows that 
$B^d_{n,r+1}:=B^d_{n,r}(\mathrm{C}_{n,r+1})$ 
has the desired primitive relations. 
\end{proof}
By Proposition \ref{blowupprim}, we can put $\Sigma^d_{n,1}:=\Sigma^d_n$. 
So, we can construct $B^d_{n,1},B^d_{n,2},B^d_{n,3},\ldots$ inductively by 
Lemma  \ref{induction} 
unless $2r+1\ge d$ or $r= n$. 
If $2r+1=d$, then $B^d_{n,r}$ has a flopping contraction, and we cannot 
obtain a smooth Famo variety. 
If $2r+1>d$, then $\widetilde{B}^d_n:=B^d_{n,r}$ is the desired 
smooth Fano variety. 

Thus, we obtain the following main theorem in this paper.
\begin{thm}\label{kekka}
The following hold$:$
\begin{enumerate}
\item 
If $2n-1<d$, then $\widetilde{B}^d_{n}:=B^d_{n,n}$ is a smooth toric Fano 
variety whose primitive relations are 
\[
x_i+y_i=0\ (1\le i\le n)\mbox{ and }
x_{n+1}+\cdots+x_{d+1}=y_1+\cdots+y_n. 
\]
$\widetilde{B}^d_{n}$ has a $(\mathbb{P}^1)^n$-bundle structre 
over $\mathbb{P}^{d-n}$. 
Moreover, $B^d_n\dashrightarrow \widetilde{B}^d_n$ is the composition 
of $2^n-n-1$ anti-flips. 
\item 
If $2n-1\ge d$ and $d$ is odd, then 
there does not exist a {\em smooth} Fano variety which is isomorphic to 
$B^d_n$ in codimension one. 
\item 
If $2n-1\ge d$ and $d$ is even, then put $c:=\frac{d}{2}$. 
In this case, $\widetilde{B}^d_n:=B^d_{n,c}$  is the desired 
smooth toric Fano variety. 
$\widetilde{B}^d_n$ has no bundle structure, and 
every extremal ray of the Kleiman-Mori cone of $\widetilde{B}^d_n$ is of small type. 
Moreover, $B^d_n\dashrightarrow \widetilde{B}^d_n$ is the composition 
of $\sum^c_{r=2}\mathrm{C}_{n,r}$ anti-flips. 
\end{enumerate}
\end{thm}
\begin{proof}
$B^d_{n,r}\dashrightarrow B^d_{n,r+1}$ is the composition of 
$\mathrm{C}_{n,r+1}$ anti-flips. Therefore, 
$B^d_n\dashrightarrow \widetilde{B}^d_n$ is the composition 
of $\sum^n_{r=2}\mathrm{C}_{n,r}$ anti-flips 
(resp. $\sum^c_{r=2}\mathrm{C}_{n,r}$) for the case (1) (resp. the case (3)). 

The case (2) means $2r+1=d$ for some $1\le r<n$. 
So, we have a flopping contraction in the middle of the operation. 
\end{proof}

\begin{rem}\label{mmprem}
In the cases (1) and (3) in Theorem \ref{kekka}, the rational map 
$B^d_n\dashrightarrow \widetilde{B}^d_n$ is a process of the so-called 
$-K_{B^d_n}$-{\em Minimal Model Program} which consists of only $-K_{B^d_n}$-flips 
(that is, anti-flips), and $\widetilde{B}^d_n$ is 
the {\em unique} $-K_{B^d_n}$-minimal model. 

The case (2) is similar. However, in this case, any $-K_{B^d_n}$-minimal model is 
not a Fano manifold (in particular, not unique). 
\end{rem}

\begin{rem}
The conditions $2n-1\ge d$ in (3) in Theorem \ref{kekka} and $n\le d+1$ 
become $c+1\le n\le 2c+1$. Thus, 
we obtain exactly $c+1$ smooth toric Fano varieties 
$\widetilde{B}^d_{c+1},\widetilde{B}^d_{c+2},\ldots,\widetilde{B}^d_{2c+1}$ 
in this case. 
\end{rem}

\begin{ex}\label{4dimexample}
We explicitly describe the $4$-dimensional operations $B^4_2\dashrightarrow \widetilde{B}^4_2$ 
and $B^4_3\dashrightarrow \widetilde{B}^4_3$. 
\begin{enumerate}
\item 
The primitive relations of $\Sigma^4_2$ are 
\[
({\rm i})\ x_1+y_1=0,\ ({\rm ii})\ x_2+y_2=0,
\]
\[
({\rm iii})\ x_2+x_3+x_4+x_5=y_1,\ ({\rm iv})\ x_1+x_3+x_4+x_5=y_2\mbox{ and }
\]
\[
({\rm v})\ y_1+y_2=x_3+x_4+x_5,
\]
where $\G(\Sigma_2^4)=\{x_1,x_2,x_3,x_4,x_5,y_1,y_2\}$. We do the anti-flip 
with respect to $({\rm v})$. Theorem \ref{mainthm} tells us that 
the primitive relations $({\rm iii})$ and $({\rm iv})$ are eliminated, since 
$\{x_3,x_4,x_5\}\subset \{x_2,x_3,x_4,x_5\}$ and 
$\{x_3,x_4,x_5\}\subset \{x_1,x_3,x_4,x_5\}$. So the primitive relations of 
the desired toric manifold $B_{2,1}^4(1)=B_{2,2}^4=\widetilde{B}_2^4$ are 
\[
({\rm i})\ x_1+y_1=0,\ ({\rm ii})\ x_2+y_2=0\mbox{ and }
\]
\[
({\rm v})^+\ x_3+x_4+x_5=y_1+y_2. 
\]
By Proposition \ref{fanohantei}, $\widetilde{B}_2^4$ is a Fano variety. 
This case corresponds to $(1)$ in Theorem \ref{kekka}.

\item 
The primitive relations of $\Sigma^4_3$ are 
\[
({\rm i})\ x_1+y_1=0,\ ({\rm ii})\ x_2+y_2=0,\ ({\rm iii})\ x_3+y_3=0,
\]
\[
({\rm iv})\ x_2+x_3+x_4+x_5=y_1,\ ({\rm v})\ x_1+x_3+x_4+x_5=y_2,
\]
\[
({\rm vi})\ x_1+x_2+x_4+x_5=y_3,\ ({\rm vii})\ y_1+y_2=x_3+x_4+x_5,
\]
\[
({\rm viii})\ y_1+y_3=x_2+x_4+x_5\mbox{ and }({\rm ix})\ y_2+y_3=x_1+x_4+x_5,
\]
where $\G(\Sigma_2^4)=\{x_1,x_2,x_3,x_4,x_5,y_1,y_2,y_3\}$. We do the $3$-times 
anti-flips 
\[
B_3^4=B_{3,1}^4(0)\dashrightarrow B_{3,1}^4(1)\dashrightarrow B_{3,1}^4(2)\dashrightarrow 
B_{3,1}^4(3) 
\]
with respect to ({\rm vii}), ({\rm viii}) and ({\rm ix}). The primitive relations of $B_{3,1}^4(1)$ 
are $({\rm i})$, $({\rm ii})$, $({\rm iii})$, $({\rm vi})$, $({\rm viii})$, $({\rm ix})$ and 
\[
({\rm vii})^+\ x_3+x_4+x_5=y_1+y_2,
\] 
since 
$\{x_3,x_4,x_5\}\subset \{x_2,x_3,x_4,x_5\}$ and 
$\{x_3,x_4,x_5\}\subset \{x_1,x_3,x_4,x_5\}$. 
The primitive relations of $B_{3,1}^4(2)$ 
are $({\rm i})$, $({\rm ii})$, $({\rm iii})$, $({\rm vii})^+$, $({\rm ix})$ and 
\[
({\rm viii})^+\ x_2+x_4+x_5=y_1+y_3,
\] 
since 
$\{x_2,x_4,x_5\}\subset \{x_1,x_2,x_4,x_5\}$. Finally, 
the primitive relations of $B_{3,1}^4(3)$ 
are $({\rm i})$, $({\rm ii})$, $({\rm iii})$, $({\rm vii})^+$, $({\rm viii})^+$, 
\[
({\rm ix})^+\ x_1+x_4+x_5=y_2+y_3\mbox{ and }({\rm x})\ y_1+y_2+y_3=x_4+x_5. 
\] 
We should remark that $\{y_1,y_2,y_3\}$ is a new primitive collection 
(see (3) in Theorem \ref{mainthm}). 
By Proposition \ref{fanohantei}, $B_{3,1}^4(3)=B_{3,2}^4=\widetilde{B}_3^4$ 
is a Fano variety, and 
this case corresponds to $(3)$ in Theorem \ref{kekka}. 
\end{enumerate}
$\widetilde{B}^4_2$ is the smooth toric Fano $4$-fold of type $D_9$, 
while 
$\widetilde{B}^4_3$ is the smooth toric Fano $4$-fold of type $M_1$ 
(see Batyrev's list \cite{bat2}). 
\end{ex}

\begin{ex}\label{3dimexample}
We consider the $3$-dimensional case $B^3_2$, that is, 
the case $(2)$ in Theorem \ref{kekka}. 
The primitive relations of $\Sigma^3_2$ are 
\[
x_1+y_1=0,\ x_2+y_2=0,\ x_2+x_3+x_4=y_1,\ x_1+x_3+x_4=y_2\mbox{ and }
y_1+y_2=x_3+x_4,
\]
where $\G(\Sigma_2^3)=\{x_1,x_2,x_3,x_4,y_1,y_2\}$. Let $B^3_2\dashrightarrow B^+$ be 
the flop with respect to $y_1+y_2=x_3+x_4$. Then, the primitive relations of $B^+$ are 
\[
x_1+y_1=0,\ x_2+y_2=0\mbox{ and }
x_3+x_4=y_1+y_2
\]
by Theorem \ref{mainthm} (see Remark \ref{flopdemo}, too). Both $B_2^3$ and $B^+$ are 
not Fano manifolds but weak Fano manifolds by Proposition \ref{fanohantei}. 
Namely, they are $-K_{B_2^3}$-minimal models for $B_2^3$. However,  
$B_2^3$ and $B^+$ are not isomorphic (see Remark \ref{mmprem}). 
\end{ex}

\begin{rem}
For an even number $d$, 
$\widetilde{B}^d_{d}$ is the pseudo-symmetric 
toric Fano variety $\widetilde{V}^d$ in \cite{ewald1}, 
while $\widetilde{B}^d_{d+1}$ is the 
symmetric toric Fano variety $V^d$ in \cite{voskre}. 
\end{rem}

\end{document}